\documentclass[leqno]{amsart}

\newcommand{\old}[1]{}

\renewcommand{\emph}[1]{\textit{#1}}
\usepackage{enumerate,amsmath,latexsym,amssymb,amsthm}
\usepackage{pdftricks}
\begin{psinputs}
  \usepackage{pstricks}
\end{psinputs}
\usepackage{color}\usepackage{graphicx}

\definecolor{brown}{cmyk}{0, 0.72, 1, 0.45}
\definecolor{grey}{gray}{0.5}

 \allowdisplaybreaks

\newcounter{rot}%\addtocounter{rot}{1}, \therot

\newcommand{\ignore}[1]{}

\def\cA{{\mathcal A}}
\def\cB{{\mathcal B}}
\def\cC{{\mathcal C}}

\def\cU{{\mathcal U}}
\def\cX{{\mathcal X}}
\def\cV{{\mathcal V}}
\def\cY{{\mathcal Y}}

\newcommand{\set}[1]{\left\{#1\right\}}

\def\cQ{\mathcal{Q}}

\def\ii_(#1,#2){i_{#1}^{#2}}

\def\a{\alpha}
\def\b{\beta}
\def\d{\delta}
\def\D{\Delta}
\def\e{\varepsilon}
\def\f{\phi}
\def\F{\Phi}
\def\g{\gamma}
\def\G{\Gamma}
\def\k{\kappa}

\def\l{\lambda}
\def\m{\mu}
\def\n{\nu}

\def\r{\rho}

\def\s{\sigma}
\def\t{\tau}
\def\om{\omega}

\def\cW{{\mathcal W}}

\newcommand{\sbs}{\subset}

\def\cD{\mathcal{D}}
\def\cE{\mathcal{E}}
\def\cF{\mathcal{F}}

\def\cT{{\mathcal T}}

\def\cS{{\mathcal S}}
\def\cO{{\mathcal O}}

\def\Re{\mathbb{R}}
\def\Z{\mathbb{Z}}
\def\N{\mathbb{N}}
\newcommand{\brac}[1]{\left( #1 \right)}

\newcommand{\expect}{\operatorname{\bf E}}
\def\E{\expect}
\renewcommand{\Pr}{\operatorname{\bf Pr}}
\newcommand\bfrac[2]{\left(\frac{#1}{#2}\right)}

\newtheorem{theorem}{Theorem}[section]
\newtheorem{conjecture}[theorem]{Conjecture}
\newtheorem{lemma}[theorem]{Lemma}
\newtheorem{corollary}[theorem]{Corollary}

{
\theoremstyle{definition}

\newtheorem{q}{}
}

\newtheorem{observation}[theorem]{Observation}
\newcounter{thmtemp}

\newcommand{\nospace}[1]{}

\def\path{\operatorname{PATH}}

\def\V{{\bf Var}}
\newcommand{\beq}[1]{\begin{equation}\label{#1}}
\def\eeq{\end{equation}}

\renewcommand{\Re}{\mathbb{R}}

\def\La{\Lambda}

\newcommand{\de}[2]{||#1-#2||}
\newcommand{\dx}[2]{\mathrm{dist}(#1,#2)}
\newcommand{\diam}{\mathrm{diam}}
\newcommand{\inv}{\mathrm{inv}}
\newcommand\gxp{\cX_{n,p}}
\newcommand\gxpr{\cX_{n,p,r}}
\newcommand\gxnp{\cX_{\N,p}}

\newcommand{\gx}[1]{\cX_{n,#1}}
\newcommand\xdn{\cX_n}
\newcommand\gyp{\cY^d_{t,p}}

\newcommand\gwp{{\mathcal W}^d_{t,p}}
\newcommand\gw[3]{{\mathcal W}^{#1}_{#2,#3}}

\newcommand\gwpa{\mathcal{W}^{d,\a}_{t,p}}
\newcommand{\gy}[3]{\cY^{#1}_{#2,#3}}
\newcommand\gypa{\cY^{d,\a}_{t,p}}
\newcommand\gypta{\cY^{d,\a}_{\t,p}}

\newcommand\ydt{\cY^d_t}
\newcommand{\yd}[1]{\cY^d_{#1}}

\newcommand{\flr}[1]{\lfloor #1 \rfloor}

\newcommand{\interval}[2]{\left[#1,#2\right]}

\begin{document}
\title{Traveling in randomly embedded random graphs}

\author{Alan Frieze}
\email[Alan Frieze]{alan@random.math.cmu.edu}
\thanks{Research supported in part by NSF grant DMS-1362785}

\author{Wesley Pegden}
\email[Wesley Pegden]{wes@math.cmu.edu}
\thanks{Research supported in part by NSF grant DMS-1363136}

\date{November 24, 2014}

\address{Department of Mathematical Sciences\\
Carnegie Mellon University\\
Pittsburgh, PA 15213\\
U.S.A.}

\begin{abstract}
We consider the problem of traveling among random points in Euclidean space, when only a random fraction of the pairs are joined by traversable connections.  In particular, we show a threshold for a pair of points to be connected by a geodesic of length arbitrarily close to their Euclidean distance, and analyze the minimum length Traveling Salesperson Tour, extending the Beardwood-Halton-Hammersley theorem to this setting.
\end{abstract}

\maketitle

\section{Introduction}
The classical Beardwood-Halton-Hammersley theorem \cite{BHH} (see also Steele \cite{S}) concerns the minimum cost Traveling Salesperson Tour through $n$ random points in Euclidean space.  In particular, it guarantees the existence of an absolute (though still unknown) constant $\b_d$ such that if $x_1,x_2\dots,$ is a random sequence of points in the $d$-dimensional cube $[0,1]^d$, the length $T(\gx 1)$ of a minimum tour through $x_1,\dots,x_n$ satisfies
\begin{equation}
\label{e.bhh}
T(\gx 1)\sim \beta_d n^{\frac{d-1}{d}}\ a.s.
\end{equation}

The present paper is concerned still with the problem of traveling among random points in Euclidean space.  In our case, however, we suppose that only a (random) subset of the pairs of points are joined by traversable connections, independent of the geometry of the point set. 

In particular, we study random embeddings of the Erd\H{o}s-R\'enyi-Gilbert random graph $G_{n,p}$ into the $d$-dimensional cube $[0,1]^d$.  We let $\cX_n$ denote a random embedding of $[n]=\{1,\dots,n\}$ into $[0,1]^d$, where each vertex $i\in [n]$ is mapped (independently) to a random point $X_i\in [0,1]^d$, and we denote by $\gxp$ the random graph whose vertex set is $\cX_n$ and whose pairs of vertices are joined by edges each with independent probability $p$.  Edges are weighted by the Euclidean distance between their points, and we are interested in the total edge-weight required to travel about the graph.

This model has received much less attention than the standard model of a random geometric graph, defined as the intersection graph of unit balls with random centers $X_i,i\in[n]$, see Penrose \cite{P}. We are only aware of the papers by Mehrabian \cite{M11} and Mehrabian and Wormald \cite{MR} who studied the {\em stretch factor} of $\gxp$. In particular, let $\de x y$ denote the Euclidean distance between vertices $x,y$, and $\dx x y$ denote their distance in $\gxp$.
They showed (considering the case $d=2$) that unless $p$ is close to 1, the stretch factor
\[
\sup_{x,y\in \gxp} \frac{\dx x y}{\de x y}
\]
tends to $\infty$ with $n$.

As a counterpoint to this, our first result shows a very different phenomenon when we pay attention to additive rather than multiplicative errors.  In particular, for $p\gg \frac{\log^d n}{n}$, the distance between a typical pair of vertices is arbitrarily close to their Euclidean distance, while for $p\ll \frac{\log^d n}{n}$, the distance between a typical pair of vertices in $\xdn$ is arbitrarily large (Figure \ref{f.paths}). 
\begin{theorem}
\label{t.expected}  Let $\om=\omega(n)\to \infty$.  We have:
\begin{enumerate}[(a)]
\item For $p\leq \frac 1 {\om^d(\log\log n)^{2d}} \frac{\log^dn}{n}$ and fixed $u,v$, 
$$\dx u v\geq \frac{\om}{8de^d} 
\qquad\text{a.a.s.}\footnote{A sequence of events $\cE_n$ occurs {\em asymptotically almost surely} (a.a.s.) if $\lim_{n\to\infty}\Pr(\neg\cE_n)=0$.}$$
\item
For $p\geq \frac{\om\log^d n}{n}$, we have a.a.s. that uniformly for all vertices $u,v$,  
\[
\dx u v =\de u v+o(1).
\]
\end{enumerate}
\end{theorem}

\begin{figure}
\hspace{\stretch{1}}\includegraphics[width=.24\linewidth]{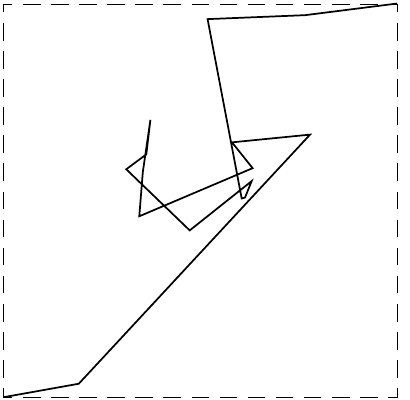}
\includegraphics[width=.24\linewidth]{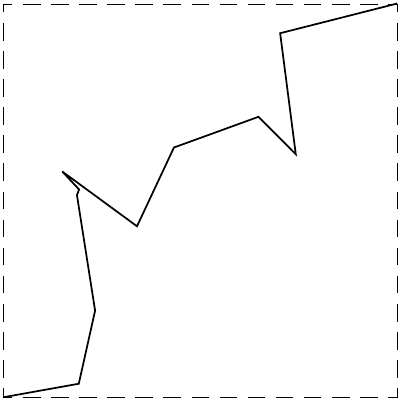}
\includegraphics[width=.24\linewidth]{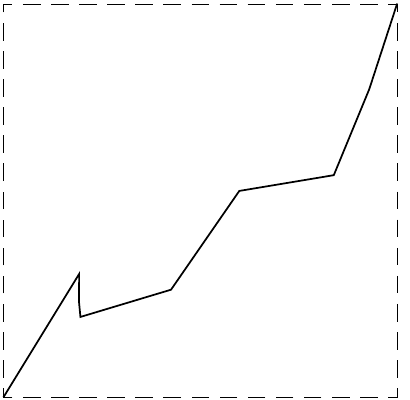}
\includegraphics[width=.24\linewidth]{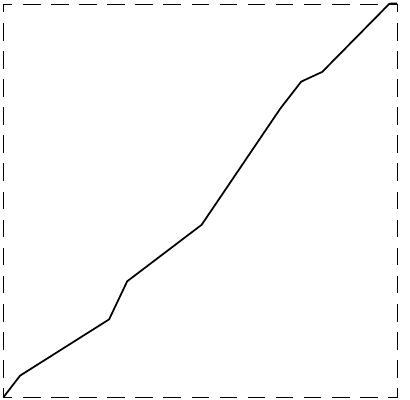}\hspace{\stretch{1}}
\caption{\label{f.paths}Paths in an instance of $\cX_{n,p}$ for $d=2$, $n=2^{30}$, and $p=\tfrac{10}{n},\tfrac{25}{n},\tfrac{50}{n},$ and $\tfrac{200}{n}$, respectively.  In each case, the path drawn is the shortest route between the vertices $x$ and $y$ which are closest to the SW and NE corners of the square. (See Q.~\ref{pathgeom}, Section \ref{Qs}.)}
\end{figure}

Theorem \ref{t.expected} means that, even for $p$ quite small, it is not that much more expensive to travel from one vertex of $\gxp$ to another than it is to travel directly between them in the plane.    On the other hand, there is a dramatic dependence on $p$ if the goal is to travel among \emph{all} points.  Let $T(\gxp)$ denote the length of a minimum length tour in $\gxp$ hitting every vertex exactly once, i.e. a Traveling Salesperson tour.  
\begin{theorem}\label{worst}
There exists a sufficiently large constant $K>0$ such that for all $p=p(n)$ such that $p\geq \frac{K\log n}{n}$, $d\geq 2$, we have that
\beq{Tgxp}
T(\gxp)=\Theta\bfrac{n^{\frac{d-1}{d}}}{p^{1/d}}\qquad a.a.s.
\eeq
\end{theorem} 
(Recall that $f(n)=\Theta(g(n))$ means that $f(n)$ is bounded between positive constant multiples of $g(n)$ for sufficiently large $n$.)  As the threshold for $G_{n,p}$ to be Hamiltonian is at $p=\frac{\log n +\log \log n+\om(n)}{n}$, this theorem covers nearly the entire range for $p$ for which a TSP tour exists a.a.s.

  Finally, we extend the asymptotically tight BHH theorem to the case of $\gxp$ for any constant $p$.  To formulate an ``almost surely'' statement, we let $\gxnp$ denote a random graph on a random embedding of $\N$ into $[0,1]^d$, where each pair $\{i,j\}$ is present as an edge with independent probability $p$, and consider $\gxp$ as the restriction of $\gxnp$ to the first $n$ vertices $\{1,\dots,n\}$.
\begin{theorem}\label{tsp}
If $d\geq2$ and $p>0$ is constant, then there exists $\b^d_p>0$ such that
\[
T(\gxp)\sim \b^d_pn^{\frac{d-1} d} \qquad a.s.
\]
\end{theorem}
Karp's algorithm \cite{K} for a finding an approximate tour through $\cX_n$ extends to the case $\gxp$, $p$ constant as well:
\begin{theorem}\label{t.alg}
For fixed $d\geq2$ and $p$ constant, then there is an algorithm that a.s. finds a tour in $\gxp$ of value $(1+o(1))\b^d_pn^{(d-1)/d}$ in polynomial time, for all $n\in \N$.
\end{theorem}

\section{Traveling between pairs}
In this section, we prove Theorem \ref{t.expected}.  Let $\nu_d$ denote the volume of a $d$-dimensional unit ball; recall that $\nu_d$ is bounded $(\nu_d\leq \nu_5<6$ for all $d$).
\begin{proof}[Proof of Theorem \ref{t.expected}(a)]
Let 
$\e=\frac{1}{\log\log n}$ and let $\cA_k$ be the event that there exists a path of length $k\geq k_0=\frac{\log n}{2d\log\log n}$ from $u$ to $v$ that uses $\leq \e k$ edges of length at least $\ell_1= \frac{\om(\log\log n)^2}{4e^d\log n}$. Then
\begin{align}
\Pr(\exists k:\cA_k)&\leq \sum_{k\geq k_0}(k-1)!\binom{n}{k-1}p^k \binom{k}{(1-\e)k}\brac{\nu_d\bfrac{\om(\log\log n)^2}{4e^d\log n}^d}^{(1-\e)k}\label{eq1}\\
&\leq  \frac{1}{n}\sum_{k\geq k_0}\brac{\frac{\nu_d\log^{d\e}n}{(4e^d)^{d(1-\e)}}\cdot \bfrac{e}{\e}^\e}^k =o(1).\nonumber
\end{align}
{\bf Explanation of \eqref{eq1}:}
Choose the $k-1$ interior vertices of the possible path and order them in $(k-1)!\binom{n}{k-1}$ ways as $(u_1,u_2,\ldots,u_{k-1})$. Then $p^k$ is the probability that the edges exist in $G_{n,p}$. Now choose the short edges $e_i=(u_{i-1},u_i),i\in I$ in $\binom{k}{(1-\e)k}$ ways and bound the probability that these edges are short by $\brac{\nu_d\bfrac{\om(\log\log n)^2}{4e^d\log n}^d}^{(1-\e)k}$ viz.~the probability that $u_i$ is mapped to the ball of radius $\ell_1$, center $u_{i-1}$ for $i\in I$. 

Now a.a.s. the shortest path in $G_{n,p}$ from $u$ to $v$ requires at least $k_0$ edges: Indeed the expected number of paths of length at most $k_0$ from $u$ to $v$ can be bounded by
$$\sum_{k=1}^{k_0}(k-1)!\binom{n}{k-1}p^k\leq \frac{1}{n}\sum_{k=1}^{k_0}\bfrac{\log^dn}{\om^d(\log\log n)^{2d}}^k=o(1).$$
So a.a.s.
$$dist(u,v)\geq \e k_0\ell_1=\frac{\e\log n}{2d\log\log n}\cdot \frac{\om(\log\log n)^2}{4e^d\log n}=\frac{\om}{8de^d}.$$
\end{proof}

\begin{proof}[Proof of Theorem \ref{t.expected}(b)]
Fix some small $\g>0$.  We begin by considering the case of vertices $u,v$ at distance $\de u v\geq \g$.  Letting $\d=\frac{1}{\log n}$, there is a constant $C$ such that, for sufficiently large $n$ relative to $\g$, we can find a set $\cB$ of $\geq \frac{2C}{\d}$ disjoint balls of radius $\d$ centered on the line from $u$ to $v$, such that $\frac{C}{\d}$ of the balls are closer to $u$ than $v$, and $\frac{C}{\d}$ balls are closer to $v$ than $u$  (Figure \ref{f.path}). Denote these two families of $\frac C {\d}$ balls by $\cF_{u,v}$ and $\cF_{v,u}$.

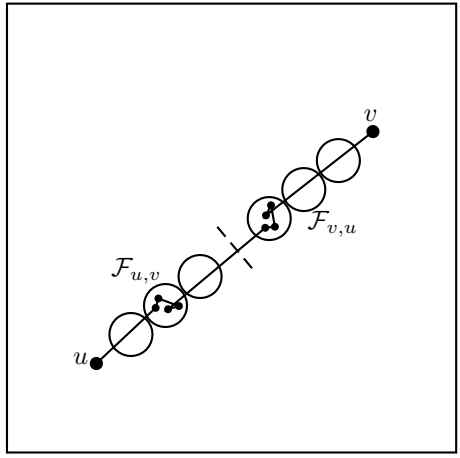
\begin{figure}[t]
\begin{center}
\begin{pdfpic}
\psset{unit=6cm,dotsize=5pt}
\begin{pspicture}(0,0)(1,1)
\psframe(0,0)(1,1)

\rput{40}(.2,.2){
\rput{-40}(-.02,.03){$u$}
\psdot(0,0)
\rput{-40}(.82,.03){$v$}
\psdot(.8,0)

\rput{-40}(.2,.1){$\mathcal{F}_{u,v}$}
\rput{-40}(.6,-.1){$\mathcal{F}_{v,u}$}

\pscircle(.1,0){.05}
\pscircle(.2,0){.05}
\pscircle(.3,0){.05}

\pscircle(.5,0){.05}
\pscircle(.6,0){.05}
\pscircle(.7,0){.05}

\psline[linestyle=dashed](.4,.06)(.4,-.06)

{\psset{dotsize=3pt}
\psdot(.18,.01)
\psdot(.198,.022)
\psdot(.222,-.02)
\psdot(.199,-.01)

\psdot(.48,-.01)
\psdot(.498,-.022)
\psdot(.522,.02)
\psdot(.499,.01)
}

\psline(0,0)(.18,.01)(.198,.022)(.222,-.02)(.199,-.01)(.48,-.01)(.498,-.022)(.522,.02)(.499,.01)(.8,0)
}
\end{pspicture}
\end{pdfpic}
\end{center}
\caption{\label{f.path}Finding a short path.}
\end{figure}

 Given a ball $B\in \cF_{\{u,v\}}=\cF_{u,v}\cup \cF_{v,u}$, the induced subgraph $G_B$ on vertices of $\cX$ lying in $B$ is a copy of $G_{N,p}$, where $N=N(B)$ is the number of vertices lying in $B$. Let 
$$\cS_B\text{ be the event that }N(B)\notin \left[\frac{N_0}{2},2N_0\right]\text{ where }N_0=\nu_d\d^d n.$$ 
The Chernoff bounds imply that for $B\in \cF_{\{u,v\}}$,
\beq{f0}
\Pr\brac{\neg\cS_B}\leq e^{-\Omega(n\d^d)}=e^{-n^{1-o(1)}}.
\eeq 
This gives us that a.a.s. $\cS_B$ holds for all pairs $u,v\in \cX$ and all $\cB$:

\begin{enumerate}[(A)]
\item All subgraphs $G_B$ for $B\in \cF_{\{u,v\}}$ have a giant component $X_B$, containing at least $N_0/3$ vertices.\\ 
Indeed, the expected average degree in $G_B$ is $Np=\Omega(\om)\to \infty$ and at this value the giant component is almost all of $B$ a.a.s.  In particular, since $\neg\cS_B$ holds, that 
\[
\Pr(\exists B:|X_B|\leq N_0/3)\leq ne^{-\Omega(N_0)}\leq ne^{-\Omega(\d^dn)}=o(1).
\]

\item \label{p.between} There is an edge between $X_B$ and $X_{B'}$ for all $B,B'\in \cF_{\{u,v\}}.$ \\
Indeed, the probability that there is no edge between $X_B,X_{B'}$, given (A), is at most\\ 
\[
(1-p)^{N_0^2/9}\leq e^{-\Omega(\d^{2d}n^2p)}\leq e^{-n^{1-o(1)}}.
\]
This can be inflated by $n^2\cdot (C\log n)^2$ to account for all pairs $u,v$ and all pairs $B,B'$.
\item \label{p.xdiam} For each $B\in \cF_{\{u,v\}}$, the graph diameter $\diam(X_B)$ (the maximum number of edges in any shortest path in $X_B$ satisfies
\[
\Pr\brac{\diam(X_B)>\frac{100\log N}{\log Np}}\leq n^{-3}.
\]
This can be inflated by $n^2\cdot (2C\log n)$ to account for pairs $u,v$ and the choice of $B\in \cF_{\{u,v\}}$. Fernholz and Ramachandran \cite{FR} and Riordan and Wormald \cite{RW} gave tight estimates for the diameter of the giant component, but we need this cruder estimate with a lower probability of being exceeded. We will prove this later in Lemma \ref{dimlem}.
\end{enumerate}

Part \eqref{p.xdiam} implies that with high probability, for any $u,v$ at distance $\geq \g$ and all $B\in \cF_{\{u,v\}}$ and vertices $x,y\in X_B$,
\begin{equation}
\label{e.bounce}
\dx x y\leq 100\d\times  \frac{\log N}{\log Np}\leq 
\frac {100} {\log n} \frac{\log n-d(\log \om+\log\log n)+O(1)}{\log \om - O(1)}=o(1).
\end{equation}
As the giant components $X_B$ ($B\in \cF_{u,v}$) contain in total at least
$\frac{C}{\d}\cdot\frac{N_0}{3}=\frac{C\nu_dn}{3\d^{d-1}}$ vertices, the probability that $u$ has no neighbor in these giant components is at most
\[
(1-p)^{\frac{C\nu_dn}{3\d^{d-1}}}\leq e^{-\frac{C\nu_dnp}{3\d^{d-1}}}=n^{-\om C\nu_d/3}.
\]
In particular, the probability is small after multiplication by $n^2$, and thus a.a.s., for all pairs $u,v\in X_{n,p}$, $u$ has a neighbor in $X_B$ for some $B\in \cF_{u,v}$ and $v$ has a neighbor in $X_{B'}$ for some $B'\in \cF_{v,u}$. Now by part \eqref{p.between} and equation \eqref{e.bounce}, we can find a path
\beq{f6}
u,w_0,w_1,\dots,w_s,z_t,z_{t-1},\dots,z_1,z_0,v
\eeq
from $u$ to $v$ where the $w_i$'s are all in some $X_B$ for $B\in \cF_{u,v}$ and the total Euclidean length of the path $w_0,\dots,w_s$ tends to zero with $n$, and the $z_i$'s are all in some $\bar X_B$ for some $B\in \cF_{v,u}$, and the total Euclidean length of the path $z_0,\dots,w_t$ tends to zero with $n$.  Meanwhile, the Euclidean segments corresponding to the three edges $u,w_0$, $w_s,z_t$, and $z_0,v$ lie within $\d$ of disjoint segments of the line segment from $u$ to $v$, and thus have total length $\leq \de u v + 6\d,$ giving
\begin{equation}
\label{e.farcase}
\dx u v\leq \de u v + 6\d+o(1)=\de u v+o(1).
\end{equation}

We must also handle vertices $u,v$ with $\de u v<\g$.  We have that 
\beq{f4}
\Pr(\exists v,B:v\text{ is not adjacent to }B)\leq n^2(1-p)^{N_0p/3}
\eeq
A fortiori, a.a.s. all vertices $u,v$ are adjacent to some vertex in any ball of radius $\g$.  In particular, we can find $w\sim u$ within distance $\tfrac 5 2 \g$ of $u$, $z\sim v$ within distance $\tfrac 5 2 \g$ of $v$, such that 
\[
\g \leq \de w z \leq 5\g,
\]
implying via \eqref{e.farcase} that 
\begin{equation}
\dx u v \leq 6\g +6\d.
\end{equation}
In particular, $\dx u v-\de u v$ is bounded by a constant which can be made arbitrarily small by making $n$ large.
\end{proof}
We complete the proof of Theorem \ref{t.expected} by proving
\begin{lemma}\label{dimlem}
Suppose that $Np=\om\to\infty, \om=O(\log N)$ and let $K$ denote the unique giant component of size $N-o(N)$ in $G_{N,p}$, that q.s.\footnote{A sequence of events $\cE_n$ occurs {\em quite surely} q.s. if $\Pr(\neg\cE_n)=O(n^{-\om(1)})$.} exists. Then for $L$ large,
$$\Pr\brac{\diam(K)\geq \frac{L\log N}{\log Np}}\leq O(N^{-L/20}).$$
\end{lemma}
\begin{proof}
Let $\cB(k)$ be the event that there exists a set $S$ of $k$ vertices in $G_{N,p}$ that induces a connected subgraph and in which more than half of the vertices have less than $\om/2$ neighbors outside $S$. Also, let $\cB(k_1,k_2)=\bigcup_{k=k_1}^{k_2}\cB_k$. Then for $k=o(N)$ we have
\begin{align}
\Pr(\cB(k))&\leq \binom{N}{k}p^{k-1}k^{k-2}2^k\brac{\sum_{i=0}^{\om/2}\binom{N-k}{i}p^i(1-p)^{N-k-i}}^{k/2}\label{sp}\\
&\leq p^{-1}(2e\om e^{-\om/3})^k\leq Ne^{-k\om/4}.\label{sp1}
\end{align}
{\bf Explanation of \ref{sp}:} $\binom{N}{k}$ bounds the number of choices for $S$. We then choose a spanning tree $T$ for $S$ in $k^{k-2}$ ways. We multiply by $p^{k-1}$, the probability that $T$ exists. We then choose half the vertices $X$ of $S$ in at most $2^k$ ways and then multiply by the probability that each $x\in X$ has at most $\om/2$ neighbors in $[N]\setminus S$.

If $\k=\k(L)=\frac{L\log N}{\log Np}$ then \eqref{sp1} implies that $\Pr(\cB(\k)\leq N^{1-L/10}$.

Next let $\cD(k)=\cD_N(k)$ be the event that there exists a set $S$ of size $k$ for which the number of edges $e(S)$ contained in $S$ satisfies $e(S)\geq 2k$. Then,
$$\Pr(\cD(k))\leq \binom{N}{k}\binom{\binom{k}{2}}{2k}p^{2k}\leq \brac{\frac{Ne}{k}\cdot\bfrac{ke\om}{2N}^2}^k= \bfrac{ke^3\om^2}{2N}^k.$$
Since $\om=O(\log n)$ we have that q.s.
\beq{sp2}
\not\exists k\in [\k(1),N^{3/4}]\text{ such that $\cD(k)$ occurs}.
\eeq

Suppose then that $\cB(k_1,k_2)\cup \cD(k_1,k_2)$ does not occur, where $k_1=\k(L/4)$ and $k_2=N^{3/4}$. Fix a pair of vertices $v,w$ and first do a breadth first search (BFS) from $v\in K$ and create sets $S_0,S_1,\dots,S_{k_1}$ where $S_i$ is the set of vertices at distance $i$ from $v$. We continue this construction unless we find that for some $i$, we have $w\in S_i$. Failing this, we must have $S_{k_1}\neq\emptyset$ and $|S_{\leq k_1}|\geq k_1$ where $S_{\leq t}=\bigcup_{i=0}^tS_i$ for $t\geq 0$. We continue this construction for $t\geq k_1$ and we see that $k_1\leq |S_{\leq t}|\leq N^{2/3}$ implies that $|S_{t+1}|\geq \om|S_t|/4$. This is because only vertices in $S_t$ have neighbors outside $S_{\leq t}$  and we have assumed that $\cB(|S_{\leq t}|)$ does not occur and because of \eqref{sp2}. Thus if $|S_{t+1}|<\om|S_t|/4$ then $S_{\leq t+1}$ has at most $\om N^{2/3}/4$ vertices and more than $\om N^{2/3}/2$ edges.

Thus if $L$ is large, then we find that there exists $t\leq k_1+\k(3/4)$ such that $|S_t|\geq N^{2/3}$. Now apply the same argument for BFS from $w$ to create sets $T_0,T_1,\ldots,T_s$, where either we reach $v$ or find that $|T_s|\geq N^{2/3}$ where $s\leq k_1+\k(3/4)$. At this point the edges between $S_t$ and $T_s$ are unconditioned and the probability there is no $S_t:T_s$ edge is at most $(1-p)^{N^{4/3}}=O(e^{-\Omega(N^{1/3})})$.
\end{proof}

\section{Traveling among all vertices}\label{TSP}
Our first aim is to prove Theorem \ref{tsp}; this will be accomplished in Section \ref{pgtsp}, below.  In fact, we will prove the following general statement, which will also be useful in the proof of Theorem \ref{worst}:

\begin{theorem}\label{gtsp}
Let $\yd{1}\sbs [0,1]^d$ denote a set of points chosen from any fixed distribution, such that the cardinality $Y=|\yd{1}|$ satisfies $\E(Y)=\mu>0$ and 
$\Pr(Y\geq k)\leq C\rho^k$ for all $k$, for some $C>0,\rho<1$  Let $\ydt$ denote a random set of points in $[0,t]^d$ obtained from $t^d$ independent copies $\yd{1}+x$ $(x\in \{0,\cdots,t-1\}^d).$

If $p>0$ is constant, $d\geq 2$, and $\gyp$ denotes the random graph on $\ydt$ with independent edge probabilities $p$, then  $\exists \b>0$ (depending on $p$ and the process generating $\yd 1$) such that
\begin{enumerate}[(i)]
\item $T(\gyp)\sim \b t^d$ a.a.s., and
\item $T(\gyp)\leq \b t^d+o(t^d)$ q.s.\footnote{In this context $O(n^{-\om(1)})$ is replaced by $O(t^{-\om(1)})$.}
\end{enumerate}
\end{theorem}

The restriction $\Pr\left(|\yd{1}|\geq k\right)\leq \rho^k$ simply ensures that we have exponential tail bounds on the number of points in a large number of independent copies of $\yd{1}$:
\begin{observation}\label{o.Ychernoff}
For the total number $T_n$ of points in $n$ independent copies of $\yd{1}$, we have
\begin{equation}\label{e.Ychernoff}
\pushQED{\qed} 
\Pr(|T_n-\mu n|>\delta \mu n)<e^{-A_\rho \delta^2 \mu^2 n}.\qedhere
\popQED
\end{equation}
\end{observation}
This is a straightforward consequence, but we do not have a reference and so we give a sketch proof in the appendix.

Note that the conditions on the distribution of $\ydt$ are satisfied for a Poisson cloud of intensity 1, and it is via this case that we will derive Theorem \ref{tsp}.  Other examples for which these conditions hold include the case where $\ydt$ is simply a suitable grid of points, or is a random subset of a suitable grid of points in $[0,t]^d$, and we will make use of this latter case of Theorem \ref{gtsp} in our proof of Theorem \ref{worst}.

Our proof is by induction on $d$.  For technical reasons (see also Question \ref{q3} of Section \ref{Qs}) Theorems \ref{gtsp} and \ref{tsp} are given just for $d\geq 2$, and before beginning with the induction, we must carry out a separate argument to bound the length of the tour in 1 dimension.
\subsection{Bounding the expected tour length in 1 dimension}
\label{s.d1}

We begin with the following simple lemma.
\begin{lemma}\label{permutations}  
Let $\sigma$ be a permutation of $[n]$, and let $\ell(\sigma)$ be $\sum_{i=1}^{n-1} |\sigma_{i+1}-\sigma_i|$. Then
\beq{invo}
\ell(\sigma)<\sigma_n+3\cdot\inv(\sigma),
\eeq
where $\inv(\sigma)$ is the number of inversions in $\sigma$.
\end{lemma}
\begin{proof}
We prove this by induction on $n$. It is trivially true for $n=1$ since in this case $\ell(\sigma)=0$. Assume now that $n>1$, and given a permutation $\sigma$ of $[n]$, consider permutation $\sigma'$ of $[n-1]$ obtained by truncation:
\[
\sigma'_{i}=\begin{cases}\sigma_i &\mbox{ if }\sigma_i<\sigma_n\\
\sigma_i-1 &\mbox{ if }\sigma_i> \sigma_n
\end{cases}
\]
We have by induction that 
\beq{induct}
\ell(\sigma')\leq \sigma'_{n-1}+3\cdot\inv(\sigma').
\eeq
Now observe that
\begin{align*}
\ell(\sigma)&=\ell(\sigma')+|\sigma_n-\sigma_{n-1}|+|\left\{i|\sigma_i<\sigma_n<\sigma_{i+1} \mbox{ OR } \sigma_i>\sigma_n>\sigma_{i+1}\right\}|\\
&\leq \ell(\sigma')+|\sigma_n-\sigma_{n-1}|+\inv(\sigma)-\inv(\sigma'),
\end{align*}
and, recalling that $\inv(\sigma)=\inv(\sigma^{-1})$,
\[
\inv(\sigma)-\inv(\sigma')=n-\sigma_n.
\]
Since $\sigma'_{n-1}\leq \sigma_{n-1}$, \eqref{induct} gives that
\begin{align*}
\ell(\sigma)&\leq\sigma_{n-1}+3\cdot\inv(\sigma') +|\sigma_n-\sigma_{n-1}|+\inv(\sigma)-\inv(\sigma')\\
&=\sigma_{n-1}+\inv(\sigma')+2(\inv(\sigma)-n+\sigma_n)+|\sigma_n-\sigma_{n-1}|+\inv(\sigma)-\inv(\sigma')\\
&=\sigma_{n-1}+3\cdot\inv(\sigma)-2n+2\sigma_n+|\sigma_n-\sigma_{n-1}|\\
&=\sigma_n+3\cdot\inv(\sigma)-(2n-\sigma_{n-1}-\sigma_n-|\sigma_n-\sigma_{n-1}|)\\
&\leq \sigma_n+3\cdot\inv(\sigma).\qedhere
\end{align*}
\end{proof}
For the 1-dimension case of Theorem \ref{tsp}, we have, roughly speaking, a 1-dimensional string of points joined by some random edges.  Lemma \ref{permutations} allows us to prove the following lemma, which begins to approximate this situation.
\begin{lemma}\label{basic}
Consider the random graph $G=G_{n,p}$ on the vertex set $[n]$ with constant $p$, where each edge $\{i,j\}\in E(G)$ is given length $|i-j|\in \mathbb{N}$. Let $Z$ denote the minimum length of a Hamilton cycle in $G$ starting at vertex 1, assuming one exists. If no such cycle exists let $Z=n^2$. Then there exists a constant $A_p$ such that 
$$\E(Z)\leq A_p n\text{ and }Z\leq \frac{2A_pn}{p},\ q.s.$$
\end{lemma}

\begin{proof}
We first write $G=G_1\cup G_2\cup G_3$ where the $G_i$ are independent copies of $G_{n,p_1}$, where $1-p=(1-p_1)^3$. We will first construct a long path in $G_1$ via the following algorithm: We start with $v_1=1$. Then for $j\geq 1$ we let 
$$\f(j)=\min_{k\in [n]}\set{k:k\notin\set{v_1,v_2,\ldots,v_j}\text{ and }v_j\sim k}$$
and let $v_{j+1}=\f(j)$ i.e. we move from $v_j$ to the lowest indexed $k$ that has not been previously visited. We repeat this until we reach $j_0$ such that $\f(j_0)$ is undefined. This defines a path $P_1$ of length  $\Lambda_1=\sum_{j=1}^{j_0-1}|v_{j+1}-v_j|$. It is convenient to extend the sequence $v_1,\ldots,v_{j_0}$ by $v_{j_0+1},\ldots,v_{n}$ where the latter is $[n]\setminus\{v_1,\ldots,v_{j_0}\}$ in increasing order. Now think of $v_1,v_2,\ldots,v_{n}$ as a permutation of $[n]$. Then Lemma \ref{permutations} implies that the length $\Lambda_1$ of the initial part corresponding to the path is at most $\ell(v)<n+3\cdot\inv(v)$.

Observe that $\Pr(j_0\leq n-k)\leq n(1-p_1)^k$. This is because at $j_0$ we find that $v_{j_0}$ has no neighbors in the set of unvisited vertices and the existence of such edges is unconditioned at this point. So, 
\beq{b1}
j_0\leq n-\frac{\log^2n}{p_1}\ q.s.
\eeq

Now let $\a_j=|\set{i<j:v_i>v_j}|,j=1,2,\ldots,n$ so that $\inv(v)=\a_1+\a_2+\cdots+\a_n$. 
Let $L_j=\max\set{v_i:1\leq i\leq j}$. Then if $i<j$ and $v_i>v_j$ we must have $j\leq v_j<v_i\leq L_j$. So,
\beq{Lj1}
\a_j\leq \D_j=L_j-j.
\eeq
Furthermore, we will need
\beq{Lj2}
|v_{i+1}-v_i|\leq |v_{i+1}-(i+1)|+|v_i-i|+1\leq \D_{i+1}+\D_i+1\qquad\text{ for }1\leq i<j_0.
\eeq
It is important therefore to analyze the sequence $\D_j,1\leq j\leq j_0$. We observe that
\beq{Lj5}
\Pr(L_{j+1}=L_j+u)\ \ \begin{cases}=1-(1-p_1)^{\D_j}&u=0.\\=p_1(1-p_1)^{\D_j+u-1}&u>0.\end{cases}.
\eeq
Furthermore, these probabilities hold regardless of previous edge exposures. This is because edges incident with $v_j$ and vertices not on $P_1$ have not been exposed. 

It will follow from \eqref{Lj5} that
\begin{align}
&\D_j\leq \frac{\log^2n}{p_1},\,\forall j,\ q.s.\label{Lj4}\\
&\E\brac{\sum_{j=1}^{j_0}\D_j}\leq \frac{n}{p_1}\label{Lj3}.\\
&\sum_{j=1}^{j_0}\D_j\leq \frac{2n}{p_1},\ q.s.\label{Lj3a}
\end{align}
We will prove \eqref{Lj4}, \eqref{Lj3}, \eqref{Lj3a}  momentarily, but first let us use them to finish the proof of the lemma.

It follows from Lemma \ref{permutations}, \eqref{Lj1} and \eqref{Lj3} that
$$\E\La_1\leq A_1n,$$ 
where $A_1=1+\frac{3}{p_1}$.

It remains to show that there is a Hamilton cycle of length not much greater then $\La_1$.  

Let $J=\set{v_{j_0+1},\ldots,v_m}$. We will use the edges of $G_2$ to insert $J$ into the path $P_1$. Let $v_j\in J_0$. Assume that $v_jj\geq n/2$, the argument for $v_j<n/2$ is similar. We examine $k=v_j-1,v_j-2,\ldots$ in turn until we find a $k$ such that (i) $(v_j,v_j-k)\in E(G_2)$, $v_j-k=v_\ell\notin J$ and (ii) $(v_j,v_{\ell-1})\in E(G_2)$. We will find such a $k$ q.s. after examining at most $\log^2n$ possibilities. Using \eqref{Lj2} and \eqref{Lj4} we see that replacing the edge $(v_{\ell-1},v_\ell)$ by a path $v_{\ell-1},v_j,v_\ell$ q.s. incorporates $v_j$ into our path at a cost of at most $O\brac{\log^2n+\frac{\log^2n}{p_1}}$ and \eqref{b1} implies that there is room to insert all vertices in $J$ in this way, without using the same $v_\ell$ more than once. This gives us a Hamilton path $x_1,x_2,\dots,x_n$ in $G_1\cup G_2$ q.s. and the total added cost over the cost of $P_1$ is q.s. $O(\log^4n)$. There is only an exponentially small probability that we cannot find $G_3$-edges $\{x_1,x_{j+1}\}$, $\{x_j,x_n\}$ which now give us a Hamilton cycle; since the maximum value of of $Z$ is just $n^2$, this gives $\E(Z)\leq A_p n$, as desired.

{\bf Proof of \eqref{Lj4}:}
First of all we note that \eqref{Lj5} that
$$\Pr\brac{\exists j: L_{j+1}\geq L_j+\frac{\log^2n}{4p_1}}\leq (1-p_1)^{\log^2n/4p_1}\leq e^{-\log^2n/4}.$$
So if there exists $j$ with $\D_j\geq \frac{\log^2n}{p_1}$ then q.s. there must be $k$ such that $\D_k\in\interval{\frac{\log^2n}{2p_1}}{\frac{3\log^2n}{4p_1}}$. But then \eqref{Lj5} implies that with probability $1-O(e^{-\log^2n/2})$, $L_{k+r}=L_k$ for $r\leq n$ and this completes the proof of \eqref{Lj4}.

{\bf Proof of \eqref{Lj3}, \eqref{Lj3a}:} It follows from \eqref{Lj5} that the sum in \eqref{Lj3} is bounded by the sum of $n$ independent geometric random variables with success probability $p_1$. This gives both the bound on expectation and the q.s. bound.
\end{proof}
We have:
\begin{corollary}\label{cor1}
Suppose that we replace the length of edge $(i,j)$ in Lemma \ref{basic} by $\xi_{i}+\cdots+\xi_{j-1}$ where $\xi_1,\xi_2,\ldots,\xi_n$ are random variables with mean bounded above by $\m$ and exponential tails. If $\xi_1,\ldots,\xi_n$ are independent of $G_{n,p}$ then $\E(Z)\leq \frac{A_p\m n}{p}$. 
\end{corollary}
\begin{proof}
The bound on the expectation follows directly from Lemma \ref{basic} and the linearity of expectation. 
\end{proof}

Let us observe now that we get an upper bound $\E(T(\gy{1}{t}{p}))\leq A_p t$ on the length of a tour in 1 dimension. We have 
\[
\E(T(\gy{1}{t}{p}))=\sum_{n=0}^{\infty}\Pr(|\gy{1}{t}{p}|=n)\E\left(\gy{1}{t}{p}\middle| |\gy{1}{t}{p}|=n\right).
\]
When conditioning on $|\gy{1}{t}{p}|=n$, we let $p_1<p_2<\cdots<p_n\sbs [0,t]$ be the points in $\gy{1}{t}{p}$.  We choose $k\in \{0,n-1\}$ uniformly randomly and let $\xi_i=||p_{k+i+1}-p_{k+i}||$, where the indices of the $p_j$ are evaluated modulo $n$.  We now have $\mu(\xi_i)\leq \frac{2t}{n}$ for all $i$, and Corollary \ref{cor1} gives that 
\[
\E\left(\gy{1}{t}{p}\middle| |\gy{1}{t}{p}|=n\right)\leq \frac{A_pn}{p}\cdot \frac{2t}{n}=O(t),
\]
and thus 
\begin{equation}
\label{d1bound}
\E\left(\gy{1}{t}{p}\right)\leq A_p t.
\end{equation}

\subsection{The asymptotic tour length}
\label{pgtsp}
Our proof of Theorem \ref{gtsp} will use recursion, by dividing the $[t]^d$ cube into smaller parts.  However, since our divisions of the cube most not cross boundaries of the elemental regions $\yd 1$, we cannot restrict ourselves to subdivisions into perfect cubes (in general, the integer $t$ may not have the divisors we like).

To this end, if $L=T_1\times T_2\times \cdots \times T_d$ where each $T_i$ is either $[0,t]$ or $[0,t-1]$, we say $L$ is a $d$-dimensional \emph{near-cube} with sidelengths in $\{t-1,t\}$.  For $0\leq d'\leq d$, we define the canonical example $L_d^{d'}:=[0,t]^{d'}\times [0,t-1]^{d-d'}$ for notational convenience, and let
\[
\Phi_p^{d,d'}(t)=\E\left(T(\gyp\cap L_d^{d'})\right).
\]
so that
$$\Phi_p^{d}(t):=\Phi_p^{d,d}(t)=\Phi_p^{d,0}(t+1).$$

In the unlikely event that $\gyp\cap L_d^{d'}$ is not Hamiltonian, we take $T(\gyp\cap L_d^{d'})= t^{d+1}\sqrt d$, for technical reasons.

Our first goal is an asymptotic formula for $\F$:

\begin{lemma}\label{l.Fasym}
There exists $\b_p>0$ such that 
\[\F^{d,d'}_p(t)\sim \beta_p t^d.\]
\end{lemma}

The proof is by induction on $d\geq 2$.  We prove the base case $d=2$ along with the general case. We begin with a technical lemma.
\begin{lemma}\label{eq10}
There is a constant $F_{p,d}>0$ such that 
\beq{Apd}
\F^{d,d'}_p(t)\leq \F^{d,d'-1}_p(t)+F_{p,d}t^{d-1}
\eeq
for all $t$ sufficiently large.  In particular, there is a constant $A_{p,d}>0$ such that 
\beq{eq100}
\F^d_p(t+h)\leq \F^d_p(t)+A_{p,d} h t^{d-1}
\eeq 
for sufficiently large $t$ and $1\leq h\leq t$.
\end{lemma}
\begin{proof}

We let $S$ denote the subgraph of $\gy d {t} p\cap L_d^{d'}$ induced by the difference $L_d^{d'}\setminus L_d^{d'-1}$.

By ignoring the $d'$th coordinate, we obtain the $(d-1)$ dimensional set $\pi(S)$, for which induction on $d$ (or line \eqref{d1bound} if $d=2$) implies an expected tour $T(S)$ of length $\F^{d-1,d'-1}_p(t)\leq \beta^{d-1}_p t^{d-1}$,
and so 
\[
\F^{d-1,d'-1}_p(t)\leq D_{p,d-1} t^{d-1}
\]
for some constant $D_{p,d-1}$, for sufficiently large $t$.

We have that
\[
\E(T(S))\leq \E(T(\pi(S))+d^{1/2}\E(|V(S)|)\leq D_{p,d-1}t^{d-1}+d^{1/2}t^{d-1}.
\]
The first inequality stems from the fact that the points in $L_d^{d'}\setminus L_d^{d'-1}$ have a $d'$ coordinate in $[t-1,t]$.

Now if $\gy d {t} p\cap L_d^{d'-1}$ and $S$ are both Hamiltonian, then we have
\beq{decompose}
T(\gy d {t} p\cap L_d^{d'})\leq T(\gy d {t} p\cap L_d^{d'-1})+T(S)+O_d(t)
\eeq
which gives us the Lemma, by linearity of expectation.  We have \eqref{decompose} because we can patch together the minimum cost Hamilton cycle in $\gy d {t} p\cap L_d^{d'-1}$ and the minimum cost path $P$ in $S$ as follows:   Let $u_1,v_1$ be the endpoints of $P$. If there is an edge $u,v$ of $H$ such that $(u_1,u),(v_1,v)$ is an edge in $\gy{d}{t}{p}$ then we can create a cycle $H_1$ through $\gy d {t} p\cap L_d^{d'-1}\cup P$ at an extra cost of at most $2d^{1/2}t$. The probability there is no such edge is at most $(1-p^2)^{t/2}$, which is negligible given the maximum value of $T(\gy d {t} p\cap L_d^{d'})$.

On the other hand, the probability that either of $\gy d {t} p\cap L_d^{d'-1}$ or $S$ is not Hamiltonian is exponentially small in $t$, which is again negligible given the maximum value of $T(\gy d {t} p\cap L_d^{d'})$.
\end{proof}
Our argument is an adaptation of that in Beardwood, Halton and Hammersley \cite{BHH} or  Steele \cite{S}, with modifications to address difficulties introduced by the random set of available edges. First we introduce the concept of a decomposition into near-cubes. (Allowing near-cube decompositions is necessary for the end of the proof, beginning with Lemma \ref{cruder}).

We say that a partition of $L_d^{d'}$ into $m^d$ near-cubes $S_\a$ with sidelengths in $\{u,u+1\}$ indexed by $\a\in [m]^d$  is a \emph{decomposition} if for each $1\leq b\leq d$, there is an integer $M_b$ such that, letting 
\[
f_b(a)=\begin{cases}
a\cdot u \mbox{ if } a<M_b\\
a\cdot u+(a-M_b) \mbox { if } a\geq M_b.
\end{cases}.
\]
we have that 
\[
S_\a=[f_1(\a_1-1),f_1(\a_1)]\times [f_2(\a_2-1),f_2(\a_2)]\times \cdots \times [f_d(\a_d-1),f_d(\a_d)].
\]
Observe that so long as $u< t^{1/2}$, $L_d^{d'}$ always has a decomposition into near-cubes with sidelengths in $\{u,u+1\}$.

First we note that tours in not-too-small near-cubes of a decomposition can be pasted together into a large tour at a reasonable cost:
\begin{lemma}
\label{paste}
Fix $\d>0$, and suppose $t=mu$ for $u=t^\g$ for $\d<\g\leq 1$ ($m,u\in \Z$), and suppose $S_\a$ $(\a\in [m]^d)$ is a decomposition of $L_d^{d'}.$  We let $\cY^{d,\a}_{t,p}:=\cY^d_{t,p}\cap S_\a$.  We have
\[
T(\gyp\cap L_d^{d'})\leq \sum_{\a\in [m]^d}T(\gypa)+4m^du\sqrt d\qquad\mbox{with probability at least}\quad 1-e^{-\Omega(u^d p)}.
\]
\end{lemma}
\begin{proof}
Let $\cB,\cC$ denote the events 
\begin{align*}
\cB&=\set{\exists \a:\cY^{d,\a}_{t,p}\text{ is not Hamiltonian}}\\
\cC&=\set{\exists \a:\left||\cY^{d,\a}_{t}|-u^d\right |\geq \d u^d},
\end{align*}
and let $\cE=\cB\cup\cC$.

Now $\Pr(\cB)\leq m^de^{-\Omega(u^dp)}$ and, by Observation \ref{o.Ychernoff}, $\Pr(\cC)\leq m^de^{-\Omega(u^d)}$ and so $\Pr(\cE)\leq e^{-\Omega(u^dp)}$. Assume therefore that $\neg\cE$ holds. Each subsquare $S_\a$ will contain a minimum length tour $H_\a$.  We now order the subcubes $\{S_\a\}$ as $T_1,\ldots,T_{m^d}$, such that for $S_\a=T_i$ and $S_\b=T_{i+1}$, we always have that the Hamming distance between $\a$ and $\b$ is 1.  Our goal is to inductively assemble a tour through the subcubes $T_1,T_2,\dots,T_j$ from the smaller tours $H_\a$ with a small number of additions and deletions of edges.

Assume inductively that for some $1\leq j<m^d$ we have added and deleted edges and found a single cycle $C_j$ through the points in $T_1,\ldots,T_j$ in such a way that (i) the added edges have total length at most $4\sqrt d ju$ and (ii) we delete one edge from $\t(T_1)$, $\t(T_j)$ and two edges from each $\t(T_i),2\leq i\leq j-1$. To add the points of $T_{j+1}$ to create $C_{j+1}$ we delete one edge $(u,v)$ of $\t(T_j)\cap C_j$ and one edge $(x,y)$ of $\t(T_{j+1})$ such that both edges $\{u,x\},\{v,y\}$ are in the edge set of $\gyp$. Such a pair of edges will satisfy (i) and (ii) and the probability we cannot find such a pair is at most $(1-p^2)^{(u^d/2-1)u^d/2}$. Thus with probability at least $1-e^{\Omega(u^d p)}$ we build the cycle $C_{m^d}$ with a total length of added edges $\leq 4\sqrt d m^d u$.
\end{proof}

Linearity of expectation (and the polynomial upper bound $t^{d+1}\sqrt d$ on $T(\gyp)$) now gives a short-range recursive bound on $\F^d_p(t)$ when $t$ factors reasonably well:
\begin{lemma}\label{lem1} For all large $u$ and $1\leq m\leq u^{10}$ $(m,u\in \N)$,
$$\Phi^d_p(mu)\leq m^d(\Phi^d_p(u)+B_{d}u)$$
for some constant $B_{d}.$\qed
\end{lemma}
\noindent Note that here we are using a decomposition of $[mu]^d$ into $m^d$ subcubes with sidelength $u$; near-cubes are not required.

To get an asymptotic expression for $\F^d_p(t)$ we now let 
$$\b=\liminf_t\frac{\F^d_p(t)}{t^d}.$$
Choose $u_0$ large and such that 
$$\frac{\F^d_p(u_0)}{u_0^d}\leq \b+\e$$ 
and then define the sequence $u_k,k\geq -1$ by $u_{-1}=u_0$ and $u_{k+1}=u_k^{10}$ for $k\geq 0$. 
Assume inductively that for some $i\geq 0$ that
\begin{equation}\label{e.inductively}
\frac{\F^d_p(u_i)}{u_i^d}\leq\b+\e+\sum_{j=-1}^{i-2}\left(\frac{A_{p,d}}{u_j}+\frac{B_{p,d}}{u_j^{d-1}}\right).
\end{equation}
This is true for $i=0$, and then for $i\geq 0$ and $0\leq u\leq u_i$ and $d\leq m\in [u_{i-1},u_{i+1}]$ we have
\begin{align}
\frac{\F^d_p(mu_i+u)}{(mu_i+u)^d}&\leq \frac{\F^d_p(mu_i)+A_{p,d}u(mu_i)^{d-1}}{(mu_i)^d}\nonumber\\
&\leq \frac{m^d(\F^d_p(u_i)+B_{p,d}u_i)+A_{p,d}u(mu_i)^{d-1}}{(mu_i)^d}\nonumber\\
&\leq \b+\e+\sum_{j=-1}^{i-2}\left(\frac{A_{p,d}}{u_j}+\frac{B_{p,d}}{u_j^{d-1}}\right) +\frac{B_{p,d}}{u_i^{d-1}}+\frac{A_{p,d}}{m}\nonumber\\
&\leq \b+\e+\sum_{j=-1}^{i-1}\left(\frac{A_{p,d}}{u_j}+ \frac{B_{p,d}}{u_j^{d-1}}\right).\label{zxc}
\end{align}
Putting $m=u_{i+1}/u_i$ and $u=0$ into \eqref{zxc} completes the induction. We deduce from \eqref{e.inductively} and \eqref{zxc} that for $i\geq 0$ we have
\beq{zxcv}
\frac{\F^d_p(t)}{t^d}\leq \b+\e+\sum_{j=-1}^{\infty}\brac{\frac{A_{p,d}}{u_j}+\frac{B_{p,d}}{u_j^{d-1}}}\leq \b+2\e\qquad\text{ for } t\in J_i=[u_{i-1}u_i,u_i(u_{i+1}+1)]
\eeq
Now $\bigcup_{i=0}^\infty J_i=[u_0^2,\infty]$ and since $\e$ is arbitrary, we deduce that
\beq{beta}
\b=\lim_{t\to\infty}\frac{\F^d_p(t)}{t^d},
\eeq
We can conclude that 
\[
\F^d_p(t)\sim \b t^d,
\]
which, together with Lemma \ref{eq10}, completes the proof of Lemma \ref{l.Fasym}, once we show that $\b>0$ in \eqref{beta}.  To this end, we let $\rho$ denote $\Pr(|\yd 1|\geq 1)$, so that $\E(|\ydt|)\geq\rho t^d$.  We say $x\in \{0,\dots,t-1\}^d$ is \emph{occupied} if there is a point in the copy $\yd 1+x$.   Observing that a unit cube $[0,1]^d+x$ $(x\in \{0,\dots,t-1\}^d)$  is at distance at least 1 from all but $3^d-1$ other cubes $[0,1]^d+y$, we certainly have that the minimum tour length through $\ydt$ is at least $\frac{\cO}{3^d-1}$, where where $\cO$ is the number of occupied $x$.  Linearity of expectation now gives that $\b>\rho/(3^d-1)$, completing the proof of Lemma \ref{l.Fasym}.

\bigskip

Before continuing, we prove the following much cruder version of Part (ii) of Theorem \ref{gtsp}:
\begin{lemma}\label{cruder}
For any fixed $\e>0$, $T(\gyp)\leq t^{d+\e}$ q.s.
\end{lemma}
\begin{proof}
We let $m=\flr{t^{1-\e/2}}$ $u=\flr{t/m}$, and let $\{\gypta\}$ be a decomposition of $\gyp$ into $m^d$ near-cubes with sidelengths in $\{u,u+1\}$. We have that q.s. each $\gypta$ has (i) $\approx u^d$ points, and (ii) a Hamilton cycle $H_\a$.  We can therefore q.s. bound all $T(\gypta)$ by $d u \cdot u^d$, and  Lemma \ref{paste} gives that q.s. $T(\gy d t p)\leq 4dut^d+4m^du\sqrt{d}.$ 
\end{proof}

To prove Theorem \ref{gtsp}, we now consider a decomposition $\{S_\a\}$ ($\a\in [m]^d$) of $\ydt$ into $m^d$ near-cubes of side-lengths in $\{u,u+1\}$, for $\g=1-\frac \e 2$, $m=\flr{t^\g},$ and $u=\flr{t/m}$.

Lemma \ref{l.Fasym} gives that
$$\E T(\gypa)\sim \b_p u^d \sim \b_p t^{(1-\g)d}.$$
Let 
\[
\cS_\g(\gyp)=\sum_{\a\in [m]^d}\min\set{T(\gypa),2dt^{(1-\g)(d+\e)}}.
\]
Note that $\cS_\g(\gyp)$ is the sum of $t^{\g d}$ identically distributed bounded random variables.

Applying Hoeffding's theorem we see that for any $t$, we have
$$\Pr(|\cS_\g(\gy d t p)-m^d \E(T(\gy d {u} p))|\geq T)\leq 
2\exp\left(-\frac{2T^2}{4m^dd^2t^{2(1-\g)(d+\e)}}\right).$$
Putting $T=t^{d\e}$ for small $\e$, we see that 
\beq{eq3}
\cS_\g(\gy d t p)=\b_p t^{d}+o(t^d)\qquad q.s.
\eeq
Now, since q.s. $T(\gypa)\leq 2dt^{(1-\g)(d+\e)}$ for all $\a$ by Lemma \ref{cruder}, we have that q.s. 
$\cS_\g(\cY^{d}_{t,p})=\sum_{\a}T(\gypa)$,
so that Lemma \ref{paste} implies that 
\beq{f1}
T(\gy d t p)\leq \cS_\g(\gy d t p)+\d_2\text{ where }\d_2=o(t^d)\qquad q.s.
\eeq

It follows from \eqref{eq3} and \eqref{f1} and the fact that $\Pr(|\ydt |=t^d)=\Omega(t^{-d/2})$ that 
\beq{f2}
T(\gyp)\leq\b_p t^d+o(t^{d})\qquad q.s.
\eeq
which proves part (ii) of Theorem \ref{gtsp}.

Of course, we have from Lemma \ref{l.Fasym} that
\begin{equation}
\E(T(\gyp))= \beta^d_p t^d+\d_1\text{ where }\d_1=o(t^d),
\end{equation}
and we show next that that this together with \eqref{f1} implies part (i) of Theorem \ref{gtsp}, that: 
\beq{f3}
T=T(\gyp)=\b_p t^d +o(t^d)\qquad a.a.s.
\eeq

We choose $0\leq\d_3=o(t^{\frac{d-1}{d}})$ such that $0\leq\d_2,|\d_1|=o(\d_3)$. Let $I=[\b t^{\frac{d-1}{d}}-\d_3,\b t^{\frac{d-1}{d}}+\d_2]$. Then we have
\begin{multline*}
\b t^{\frac{d-1}{d}}+\d_1=\E(T(\gyp)\mid T(\gyp)\geq \b t^{\frac{d-1}{d}}+\d_2)\Pr(T(\gyp)\geq \b t^{\frac{d-1}{d}}+\d_2)\\
+\E(T(\gyp)\mid T(\gyp)\in I)\Pr(T(\gyp)\in I)+\\
\E(T(\gyp)\mid T(\gypa)\leq \b t^{\frac{d-1}{d}}-\d_3)\Pr(T(\gyp)\leq \b t^{\frac{d-1}{d}}-\d_3).
\end{multline*}
Now $\e_1=\E(T(\gyp)\mid T(\gyp)\geq \b t^{\frac{d-1}{d}}+\d_2)\Pr(T(\gyp)\geq \b t^{\frac{d-1}{d}}+\d_2)=O(t^{-\om(1)})$ since $|\gyp|\leq 2d^{1/2}t^d$ 
and $\Pr(T(\gyp)\geq \b t^{\frac{d-1}{d}}+\d_2)=O(t^{-\om(1)})$.

So, if $\l=\Pr(T(\gyp)\in I)$ then we have
$$\b t^{\frac{d-1}{d}}+\d_1\leq \e_1+(\b t^{\frac{d-1}{d}}+\d_2)\l+(\b t^{\frac{d-1}{d}}-\d_3)(1-\l)$$
or
$$\l\geq \frac{\d_1-\e_1+\d_3}{\d_2+\d_3}=1-o(1),$$
and this proves \eqref{f3} competing the proof of Theorem \ref{gtsp}.\qed

To derive Theorem \ref{tsp}, we now let $\gwp$ be the graph on the set of points in $[0,t]^d$ which is the result of a Poisson process of intensity 1.  Our task is now to control the variance of $T(\gwp)$.  Here we follow Steele's argument \cite{S} with only small modifications. 

Let $\cE_t$ denote the event that 
\[
T(\gw d {2t} p)\leq \sum_{\a\in [2]^d} T(\gwpa) +2^{d+2}t\sqrt d.
\]
Observe that Lemma \ref{paste} implies that 
\begin{equation}\label{pneg}
\Pr(\neg\cE_t)\leq e^{-\Omega(t^d p)}.
\end{equation}

We define the random variable $\lambda(t)=T(\gw d {t} p)+10\sqrt d t,$ and let $\lambda_i$ denote independent copies.  Conditioning on $\cE_t$, we have
\begin{equation}\label{ideq}
\lambda_0 (2t)\leq \sum_{i=1}^{2^d}\lambda_i(t)-4\sqrt d t \leq \sum_{i=1}^{2^d}\lambda_i(t).
\end{equation}
In particular, \eqref{pneg} implies that there is enough room that, letting $\Upsilon(t)=\E(\lambda(t))$ and $\Psi(t)=\E(\lambda(t)^2)$, we have for sufficiently large $t$ that
\[
\Psi(2t)\leq 2^d\Psi(t)+2^d(2^d-1)\Upsilon^2(t)
\]
and for 
\[
\cV(t):=\V(T(\gw d {t} p))=\Psi(t)-\Upsilon(t)^2,
\]
we have
\[
\frac{\cV(2t)}{(2t)^{2d}}-\frac{1}{2^d}\frac{\cV(t)}{(t)^{2d}}\leq \frac{\Upsilon^2(t)}{t^{2d}}-\frac{\Upsilon^2(2t)}{(2t)^{2d}}.
\]
Now summing over $t=2^kt_0$ for $k=0,\dots,M-1$ gives
\[
              \sum_{k=1}^M\frac{\cV(2^kt)}{(2^kt)^{2d}}-
\frac 1 {2^d} \sum_{k=0}^{M-1} \frac{\cV(2^kt)}{(2^kt)^{2d}}\leq
\frac{\Upsilon^2(t)}{t^{2d}}-\frac{\Upsilon^2(2^M t)}{(2^M t)^{2d}}\leq \frac{\Upsilon^2(t)}{t^{2d}}
\]
and so, solving for the first sum, we find
\begin{equation}
\label{varsum}
\sum_{k=1}^M\frac{\cV(2^kt)}{(2^kt)^{2d}}\leq
(1-\frac{1}{2^d})\left(\frac{\cV(t)}{t^{2d}}+\frac{\Upsilon^2(t)}{t^{2d}}\right)<\infty.
\end{equation}
Still following Steele, we let $N(t)$ be the Poisson counting process on $[0,\infty).$  We fix a random embedding $\cU$ of $\N$ in $[0,1]^d$ as $u_1,u_2,\dots$ and a random graph $\cU_{p}$ where each edge is included with independent probability $p$.  We let $\cU_{n,p}$ denote the restriction of this graph to the first $n$ natural numbers.  In particular, note that $\cU_{N(t^d),p}$ is equivalent to $\cW_{t,p}$, scaled from $[0,t]^d$ to $[0,1]^d$.  Thus, applying Chebychev's inequality to \eqref{varsum} gives that
\begin{equation}
\sum_{k=0}^\infty\Pr\left(\left|\frac{t2^k T(\cU_{N((t2^k)^d),p})}{(t2^k)^d}-\beta^d_p\right|>\e\right)<\infty
\end{equation}
and so for $t>0$ that
\begin{equation}\label{2powlim}
\lim_{k\to \infty}\frac{T(\cU_{N((t2^k)^d),p})}{(t2^k)^{d-1}}=\beta\qquad a.s.
\end{equation}
Now choosing some large integer $\ell$, we have that \eqref{2powlim} holds simultaneously for all the (finitely many) integers $t\in S_P=[2^\ell,2^{\ell+1})$; and $r\in \Re$, we have that $r\in [2^kt,2^k(t+1))$ for $t\in S_\ell$ and some $k$.

Unlike the classical case $p=1$, in our setting, we do not have monotonicity of $T(\cU_{n,p})$.  Nevertheless, we show a kind of continuity of the tour length through $T(\cU_{n,p})$:
\begin{lemma}\label{sandwich}
For all $\e>0$, $\exists \d>0$ such that for all $0\leq k<\d n$, we have 
\begin{equation}
\label{cont}
T(\cU_{n+k,p})<T(\cU_{n,p})+\e n^{\frac{d-1}{d}},\qquad q.s.
\end{equation}
\end{lemma}
\begin{proof}
We consider cases according to the size of $k$.  

\noindent\textbf{Case 1}: $k\leq n^{\frac 1 3}$.\\
Note that we have $T(\cU_{n+1,p})<T(\cU_{n,p})+\sqrt d$ q.s., since we can q.s. find an edge in the minimum tour though $\cU_{n,p}$ whose endpoints are both adjacent to $(n+1)$.   $n^{\frac 1 3}$ applications of this inequality now give \eqref{cont}.

\noindent \textbf{Case 2:} $k> n^{\frac 1 3}$. \\ 
In this case the restriction $\mathcal{R}$ of $\cU_{n+k,p}$ to $\{n+1,\dots,k\}$ is q.s. (with respect to $n$) Hamiltonian \cite{BFF}.  In particular, by Theorem \ref{gtsp}, we can q.s. find a tour $T$ though $\mathcal{R}$ of length $\leq 2\beta^d_p k^{\frac{d-1}d}$.  Finally, there are q.s., edges $\{x,y\}$ and $\{w,z\}$ on the minimum tours through $\cU_{n,p}$ and $\mathcal{R}$, respectively, such that $x\sim w$ and $y\sim z$ in $\cU_{n+k,p}$, giving a tour of length
\[
T(\cU_{n+k,p})\leq T(\cU_{n,p})+ 2\beta^d_p k^{\frac{d-1}d}+4\sqrt d.\qedhere
\]
\end{proof}
Applying Lemma \ref{sandwich} and the fact that $N((1+\d)r^d)<(1+2\d)N(r^d)$ q.s (with respect to $r$). gives that for some $\e_\ell>0$ which can be made arbitrarily small by increasing $\ell$, we have q.s.
\[
T(\cU_{N(((t+1)2^k)^d),p})-\e_\ell r^{d-1}
<
T(\cU_{N(r^d),p})
<
T(\cU_{N((t2^k)^d),p})+\e_\ell (t2^k)^{d-1},
\]
and so dividing by $r^{d-1}$ and taking limits we find that a.s.
\[
(\beta-\e_\ell)(1+\tfrac{1}{2^p})^{d-1}
\leq 
\liminf_{r\to \infty} \frac{T(\cU_{N(r^d)})}{r^{d-1}}
\leq
\limsup_{r\to \infty} \frac{T(\cU_{N(r^d)})}{r^{d-1}}
\leq
\frac{\beta+\e_\ell}{(1+\frac{1}{2^p})^{d-1}}.
\]
Since $\ell$ may be arbitrarily large, we find that 
\[
\lim_{r\to \infty} \frac{T(\cU_{N(r^d)})}{r^{d-1}}=\beta.
\]
Now the elementary renewal theorem guarantees that 
\[
N^{-1}(n)\sim n,\qquad a.s.
\]
So we have a.s.
\[
\lim_{r\to \infty} \frac{T(\cU_{n,p})}{n^{\frac{d-1}{d}}}= \lim_{r\to \infty} \frac{T(\cU_{N(N^{-1}(n)),p})}{(N^{-1}(n))^{\frac{d-1}{d}}}\frac {(N^{-1}(n))^{\frac{d-1}{d}}}{n^{\frac{d-1}{d}}}=\beta\cdot 1=\beta.
\]

\subsection{The case $p(n)\to 0$}\label{TSPworst}
We will in fact show that \eqref{Tgxp} holds q.s. for $np\geq \om\log n$, for some  $\om\to\infty$. That we also get the statement of Theorem \ref{TSPworst} can be seen by following the proof carefully, but this also follows as a consequence directly from the appendix in Johannson, Kahn and Vu \cite{JKV}.

We first show that q.s. 
\beq{W1}
T(\gxp)=\Omega(n^{(d-1)/d}/p^{1/d}).
\eeq 
Let $Y_1$ denote the number of vertices whose closest $G_{n,p}$-neighbor is within $\frac{1}{(np)^{1/d}}$. Observe first that if $r=1/(np)^{1/d}$ then with probability $\geq \brac{1-\n_d r^dp}^{n-1}\approx e^{-\n_d}$, there are no points within distance $1/(np)^{1/d}$ of any fixed $v\in\gxp$. Thus $\E(Y_1)\geq ne^{-\n_d}/2$ and one can use the Azuma-Hoeffding inequality to show that $Y_1$ is concentrated around its mean. Thus q.s. $T(\gxp)\geq n^{(d-1)/d}e^{-\n_d}/4p^{1/d}$, proving \eqref{W1}.

We will for convenience prove
\begin{theorem}\label{convenience}
Let $\yd{1}\sbs [0,1]^d$ denote a set of points chosen via a Poisson process of intensity one in $[0,t]^d$ where $t=n^{1/d}$. Then there exists a constant $\g_p^d$ such that
$$T(\gyp)\leq \g_p^d \frac{t^d}{p^{1/d}}\qquad q.s.$$
\end{theorem}
\begin{proof}
We consider independent copies  of $\cY^d_{t,p_i},\,i=1,2,\ldots,k+1$. We will let $p_0=p_1=p/3$ and $p_i=p_1/2^{i},i=1,2,\ldots,k=\log_2t$ and define $p_{k+1}$ so that $1-p=\prod_{j=1}^{k+1}(1-p_j)$.  Observe that with this choice, we have that $\gyp$ decomposes as $\gyp=\bigcup_{i=0}^{k+1}G_i$, where the $G_i$ are spanning subgraphs given by independent instances of $\cY^d_{t,p_i}$.

We continue by constructing a large cycle, using only the edges of $G_1$. We choose $\e$ small and then choose $K$ sufficiently large for subsequent claims. In preparation for an inductive argument we let $t_1=t$, $T_1=t_1^d$, $m_1=\flr{(T_1p_1/K)^{1/d}}$ and consider the partition $\D_1=\{S_\a\}$ ($\a\in [m_1]^d$) of $[0,t]^d$ into $m_1^d$ subcubes of side length $u=t/m$. (Note that $t$ will not change throughout the induction). Now each $S_{\a}$ contains $\approx K/p_1$ vertices, in expectation and so it has at least $(1-\e)K/p_1$ vertices with probability $1-e^{-\Omega(K/p_1)}=1-o(1)$. Let $\a$ be {\em heavy} if $S_\a$ has at least this many vertices, and {\em light} otherwise. Let $\G_\a$ be the subgraph of $G_{1}$ induced by $S_\a$. If $\a$ is heavy then for any $\e>0$ we can if $K$ is sufficiently large find with probability at least $1-e^{-\Omega(K/p_1)}=1-o(1)$, a cycle $C_\a$ in $\G_\a$ containing at least $(1-\e)^2K/p_1$ vertices. This is because when $\a$ is heavy, $\G_\a$ has expected average degree at least $(1-\e)K$.  We say that a heavy $\a$ is {\em typical} if it $\G_\a$ contains a cycle with $(1-\e)|S_\a\cap\cX|$ edges; otherwise it is {\em atypical}.

We now let $N$ denote the set of vertices in $\bigcup C_\a$, where the union is taken over all typical heavy $\a$.  Our aim is to use Theorem \ref{gtsp}(ii) to prove that we can q.s.~merge the vertices $N$ into a single cycle $C_1$, without too much extra cost, and using only the edges of $G_1$.  Letting $q_\a=Pr(S_\a\text{ is normal})\geq 1-\e$, we make each typical heavy $\a$ \emph{available} for this round with independent probability $\frac{1-\e}{1-q_\a}$, so that the probability that any given $\a$ is available is exactly $1-\e$. (This is of course {\em rejection sampling}.) Now we can let $Y=\yd 1$ in Theorem \ref{gtsp} be a process which places a single point at the center of $[0,1]^d$ with probability $1-\e$, or produces an empty set with probability $\e$. Let now $Y_\a$ ($\a\in t^d)$ be the independent copies of $Y$ which give $\ydt$. Given two cycles $C_1,C_2$ in a graph $G$ we say that edges $u_i=(x_i,y_i)\in C_i,i=1,2$ are a {\em patchable pair} if $f_x=(x_1,x_2)$ and $f_y=(y_1,y_2)$ are also edges of $G$.  Given $x\in Y_\a, y\in Y_{\b}$, we let $x\sim y$ whenever there exist \emph{two disjoint} patchable pairs $\s_{\a,\b}$ between $C_\a, C_\b$.  Observe that an edge between two vertices of $\yd 1$ is then present with probability 
\[
q_{\a,\b}\geq\Pr(Bin(K^2/100p_1^2,p_1^2)\geq 2)\geq 1-\e.  
\]
In particular, this graph contains a copy of $\cY^d_{1,(1-\e)}$, for which Theorem \ref{gtsp}(ii) gives that q.s. we have a tour of length $\leq B_1 m_1^d$ for some constant $B_1$; in particular, there is a path $P=(\a_1,\a_2,\dots,\a_M)$ through the typical heavy $\a$ with at most this length. Using $P$, we now merge its cycles $C_{\a_i},i=1,2,\ldots,M$ into a single cycle.

Suppose now that we have merged $C_{\a_1},C_{\a_2},\ldots,C_{\a_j}$ into a single cycle $C_j$ and have used one choice from $\s_{\a_{j-1},\a_j}$ to patch $C_{\a_j}$ into $C_{j-1}$. We initially had two choices for patching $C_{\a_{j+1}}$ into $C_{\a_j}$, one may be lost, but one at least will be available. Thus we can q.s. use $G_1$ to create a cycle $H_1$ from $C_{\a_1},C_{\a_2},$ by adding only patchable pairs of edges, giving a total length of at most
\beq{cy1}
2T_1\times \frac{2t_1d^{1/2}}{m_1}+a_1m_1^d\times \frac{2t_1d^{1/2}}{m_1}\leq \frac{3T_1d^{1/2}}{p_1^{1/d}}.
\eeq
The first term in \eqref{cy1} is a bound on the total length of the cycles $C_\a$ where $\a$ is available, assuming that $|\gyp|\leq 2t^d$. The second smaller term is the q.s. cost of patching these cycles into $H_1$.

Having constructed $H_1$, we will consider coarser and coarser subdivisions $\cD_i$ of $[0,t]^d$ into $m_i^d$ subcubes, and argue inductively that we can q.s. construct, for each $1\leq i\leq \ell$ for suitable $\ell$, vertex disjoint cycles $H_1,H_2,\ldots,H_\ell$ satisfying:
\begin{enumerate}
\item \label{P.leftover} $T_i\leq 3\e T_{i-1}$ for $i\geq 2$, where $T_j=t^d-\sum_{i=1}^{j-1}|H_i|$,
\item \label{P.ind} the set of points in the $\a$th subcube in the decomposition $\cD_i$ occupied by vertices which fail to participate in $H_i$ is given by a process which occurs independently in each subcube in $\cD_i$, and
\item \label{P.length}  the total length of each $H_i$ is at most $\frac{3T_id^{1/2}}{p_i^{1/d}}$. 
\end{enumerate}
Note that $H_1$, above, satisfies these conditions for $\ell=1$.

Assume inductively that we have constructed such a sequence $H_1,H_2,\ldots,H_{j-1}$ $(j\geq 2)$. We will now use the $G_j$ edges to construct another cycle $H_j$. Suppose now that the set $\cT_j$ of points that are not in $\bigcup_{i=1}^{j-1}H_i$ satisfies $T_{j}=|\cT_j|\geq t^{d-1}/\log t$. We let $m_j=(T_{j}p_j/K)^{1/d}$ and $t_j=T_j^{1/d}$. The expected number of points in a subcube will be $K/p_j$ but we have not exercised any control over its distribution. For $i\geq 2$, we let $\a\in [m_i]^d$ be heavy if $S_\a$ contains at least $\e K/p_j$ points. Now we want $K$ to be large enough so that $\e K$ is large and that a heavy subcube has a cycle of size $(1-\e)|\cT_j\cap S_\a|$ with probability at least $1-\e$, in which case, again, it is \emph{typical}. We define $\G_j$ as the set of typical heavy pairs $\{\a,\b\}$ for which there are at least two disjoint patchable pairs between the corresponding large cycles.  Applying the argument above with $T_j,t_j,m_j,\G_j$ replacing $T_1,t_1,m_1,\G_1$ (note that \ref{P.ind}, above, ensures that Theorem \ref{gtsp} applies) we can q.s. find a cycle $H_j$ with at least $(1-3\e)T_j$ vertices and length at most $\frac{3T_jd^{1/2}}{p_j^{1/d}}$, giving induction hypothesis part \ref{P.length}.    Part \ref{P.leftover} is satisfied since the light subcubes only contribute $\e$ fraction of points to $\cT_j$, and we q.s. take a $(1-\e)$ fraction of the heavy subcubes.   Finally, Part \ref{P.ind} is satisfied since participation in $H_j$ is determined exclusively by the set of adjacency relations in $G_j\cap \cT_j$, which is independent of the positions of the vertices.

Thus we are guaranteed a sequence $H_1,H_2,\ldots, H_\ell$ as above, such that $T_{\ell+1}<t^{d-1}/\log t$. The total length of $H_1,H_2,\ldots,H_\ell$ is at most
\beq{almost}
\sum_{i=1}^\ell \frac{3T_id^{1/2}}{p_i^{1/d}}\leq \frac{3^{1+1/d}t^d}{p^{1/d}}\sum_{i=1}^\infty 3^i\cdot 2^{i/d} \e^{i-1}=O\bfrac{t^d}{p^{1/d}}.
\eeq
We can now use $G_0$ to finish the proof. It will be convenient to write $G_0=\bigcup_{i=0}^2A_i$ where $A_i,i=1,2,3$ are independent copies of $\cY^d_{t,q}$ where $1-p_0=(1-q)^3$. Also, let $R=\set{x_1,x_2,\ldots,x_r}=\gyp\setminus\bigcup_{i=1}^\ell H_i$.

We first create a Hamilton path containing all vertices, only using the edges of $A_1\cup A_2$ and the extension-rotation algorithm introduced by P\'osa \cite{Po}. We begin by deleting an arbitrary edge from $H_1$ to create a path $P_1$. Suppose inductively that we have found a path $P_j$ through $Y_j=H_1\cup \cdots H_{\r_j}\cup  X_{j}$ where $X_{j}\subseteq R$ at an added cost of $O(jt)$. We let $V_j$ denote the vertices of $P_j$ and promise that $V_{\ell+r}=\gyp$. We also note that $|V_j|\geq |V_1|=\Omega(t^d)$ for $j\geq 1$.

At each stage of our process to create $P_{j+1}$ we will construct a collection $\cQ=\set{Q_1,Q_2,\ldots,Q_r}$ of paths through $V_j$. Let $Z_\cQ$ denote the set of endpoints of the paths in $\cQ$. {\em Round} $j$ of the process starts with $P_j$ and is finished when we have constructed $P_{j+1}$.

If at any point in round $j$ we find a path $Q$ in $\cQ$ with an endpoint $x$ that is an  $A_2$-neighbor of a vertex in $y\notin V_{j}$ then we will make a {\em simple extension}
and proceed to the next round. If $x\in H_i$ then we delete one of the edges in $H_i$ incident with $y$ to create a path $Q'$ and then use the edge $(x,y)$ to concatenate $Q,Q'$ to make $P_{j+1}$. If $x\in R$ then $P_{j+1}=Q+y$. 

If $Q=(v_1,v_2,\ldots,v_s)\in \cQ$ and $(v_s,v_1)\in A_1$ then we can take any $y\notin V_j$ and with probability at least $1-(1-q)^s=1-O(t^{-\om(1)})$ find an edge $(y,v_i)\in A_2$. If there is a cycle $H_i$ with $H_i\cap V_j=\emptyset$ then we choose $y\in H_i$ and delete one edge of $H_i$ incident with $y$ to create a path $Q'$ and then we can take $P_{j+1}=(Q',v_i,v_{i-1},\ldots,v_{i+1})$ and proceed to the next round. Failing this, we choose any $y\in R\setminus V_j$ and let $P_{j+1}=(y,v_i,v_{i-1},\ldots,v_{i+1})$ and proceed to the next round. Note that this is the first time we will have examined the $A_2$ edges incident with $y$. We call this a {\em cycle extension}.

Suppose now that $Q=(v_1,v_2,\ldots,v_s)\in \cQ$ and $(v_s,v_i)\in A_1$ where $1<i<s-1$. The path $Q'=(v_1,\ldots,v_i,v_s,v_{s-1},\ldots,v_{i+1})$ is said to be obtained by a rotation. $v_1$ is the {\em fixed} endpoint. We partition $\cQ=\cQ_0\cup \cQ_1\cup\cdots\cup \cQ_{k_0},k_0=\log t$ where $\cQ_0=\set{P_j}$ and $\cQ_i$ is the set of paths that are obtainable from $P_j$ by exactly $i$ rotations with fixed endpoint $v_1$. We let $N_i$ denote the set of endpoints of the paths in $\cQ_i$, other than $v_1$, and let $\n_i=|N_i|$ and let $N_\cQ=\bigcup_iN_i$. We will prove that q.s. 
\beq{posa}
|\n_i|\leq \frac{1}{100q}\text{ implies that }|\n_{i+1}|\geq \frac{|\n_i|t^dq}{300}.
\eeq
It follows from this that q.s. we either end the round through a simple or cycle extension or arrive at a point where the paths in $\cQ$ have $\Omega(t^d)$ distinct endpoints. We can take an arbitrary $y\notin V_j$ and find an $A_2$ neighbor of $y$ among $N_\cQ$. The probability we cannot find a neighbor is at most $(1-q)^{\Omega(t^d)}=O(t^{-\om(1)})$. Once we prove \eqref{posa} we will have shown that we can create a Hamilton path through $\gyp$ from $H_1,H_2,\ldots,H_\ell,R$ at an extra cost of $O(d^{1/2}(t\ell+t^{d-1}/\log t\times \log t\times t))=O(t^d)$. We will not have used any $A_3$ edges to do this. The second $\log t$ factor comes from the fact that each path is obtained by at most $k_0$ rotations and each rotation adds one new edge.

{\bf Proof of \eqref{posa}:} We first prove that in the graph induced by $A_1$ we have
\beq{posa1}
|S|\leq  \frac{1}{100q}\text{ implies that }|N_{A_1}(S)|\geq \frac{|S|t^{d}q}{100}.
\eeq
Here $N_{A_1}(S)$ is the set of vertices not in $S$ that have at least one $A_1$-neighbor in $S$.

Indeed, if $s_0=\frac{1}{100q}=o(n)$ then
\begin{align*}
\Pr(\exists S)&\leq \sum_{s=1}^{s_0}\binom{t^d}{s}\Pr\brac{Bin(t^d-s,1-(1-q)^s)\leq \frac{s t^{d}q}{100}}\\
&\leq \sum_{s=1}^{s_0}\binom{t^d}{s}\Pr\brac{Bin\brac{t^d-s,\frac{sq}{2}}\leq \frac{s t^{d}q}{100}}\\
&\leq \sum_{s=1}^{s_0}\brac{\frac{t^de}{s}\cdot e^{-\Omega(t^{d}q)}}^s\\
&=O(t^{-\om(1)}).
\end{align*}
Now \eqref{posa} holds for $i=0$ because q.s. each vertex in $\gyp$ is incident with at least $t^{d}q/2$ $A_1$ edges. Given \eqref{posa1} for $i=0,1,\ldots,i-1$ we see that $\n_1+\cdots+\n_{i-1}=o(\n_i)$. In which case \eqref{posa1} implies that
$$\n_{i+1}\geq \frac{|N_{A_1}(N_i)|-(\n_0+\cdots+\n_{i-1})}{2}\geq \frac{t^{\g d}\n_i}{2+o(1)}$$ 
completing an inductive proof of \eqref{posa}.

Let $P^*$ be the Hamilton path created above. We now use rotations with $v_1$ fixed via the edges $A_2$ to create $\Omega(t^d)$ Hamilton paths with distinct endpoints. We then see that q.s. one of these endpoints is an $A_2$-neighbor of $v_1$ and so we get a tour at an additional cost of $O(d^{1/2}t)$.

This completes the proof of Theorem \ref{convenience}.
\end{proof}

The upper bound in Theorem \ref{worst} follows as before by (i) replacing $\gyp$ by $\gxp^d$, allowable because our upper bound holds q.s. and $\Pr(|\gyp|=t^d)=\Omega(t^{-d/2})$ and then (ii) scaling by $n^{-1/d}$ so that we have points in $[0,1]^d$.

\section{An algorithm}

To find an approximation to a minimum length tour in $\gxp$, we can use a simple version of Karp's algorithm \cite{K}. We let $m=(n/K\n_d\log n)^{1/d}$ for some constant $K>0$ and partition $[0,1]^d$ into $m^d$ subcubes of side $1/m$, as in Lemma \ref{paste} . The number of points in each subsquare is distributed as the binomial $B(n,q)$ where $q=K\log n/n$ and so we have a.a.s.~that every subsquare has $K\log n\pm\log n$, assuming $K$ is large enough. The probability that there is no Hamilton cycle in $S_{\a}$ is $O(e^{-Knqp/2})$ and so a.a.s. every subsquare induces a Hamiltonian subgraph. Using the dynamic programming algorithm of Held and Karp \cite{HK} we solve the TSP in each subsquare in time $O(\s^22^\s)\leq n^K$, where $\s=\s_{\a}=|S_\a\cap \cX_{n,p}|$. Having done this, we can with probability of failure bounded by $m^2(1-p^2)^{(K\log n)^2}$ patch all of these cycles into a tour at an extra $O(m^{d-1})=o(n^{\frac{d-1}{d}})$ cost. The running time of this step is $O(m^d\log^2n)$ and so the algorithm is polynomial time overall. The cost of the tour is bounded q.s.~as in Lemma \ref{paste}. This completes the proof of Theorem \ref{tsp}.

\section{Further questions}
\label{Qs}
Theorem \ref{t.expected} shows that there is a definite qualitative change in the diameter of $\gxp$ at around $p=\frac{\log^dn}{n}$, but our methods leave a $(\log\log n)^{2d}$ size gap for the thresholds.
\begin{q}
What is the precise threshold for there to be distances in $\gxp$ which tend to $\infty$?   What is the precise threshold for distance in $\gxp$ to be arbitrarily close to Euclidean distance?  What is the behavior of the intermediate regime?
\end{q}

\noindent One could also analyze the geometry of the geodesics in $\gxp$ (Figure \ref{f.paths}).  For example:
\begin{q}
\label{pathgeom}
Let $\ell$ be the length of a random edge on the geodesic between fixed points at at constant distance in $\gxp$.  
What is the distribution of $\ell$?
\end{q}

\smallskip
Improving Theorem \ref{worst} to give an asymptotic formula for $T(\gxp)$ is another obvious target.  It may seem unreasonable to claim such a formula for all (say, decreasing) functions $p$; in particular, in this case, the constant in the asymptotic formula would necessarily be universal.  The following, however, seems reasonable:

\begin{conjecture} If $p=\frac 1 {n^\a}$ for some constant $0<\a<1$ then there exists a constant $\b^d_\a$ such that a.a.s. $T(\gxp)\sim \b_\a\frac{n^{\frac{d-1}{d}}}{p^{1/d}}$. 
\end{conjecture}

We note that $T(\gx 1)$ is known to be remarkably well-concentrated around its mean; see, for example, the sharp deviation result of Rhee and Talagrand \cite{RT}.
\begin{q}
How concentrated is the random variable $T(\gxp)$?
\end{q}
The case of where $p=o(1)$ may be particularly interesting.

\bigskip Even for the case $p=1$ covered by the BHH theorem, the constant $\b_1^d$ $(d\geq 2)$ from Theorem \ref{gtsp} is not known.  Unlike the case of $p=1$, the 1-dimensional case is not trivial for our model.  In particular, we have proved Theorems \ref{tsp} and \ref{worst} only for $d\geq 2$.  We have ignored the case $d=1$ not because we consider the technical problems insurmountable, but because we hope that it may be possible to prove a stronger result for $d=1$, at least for the case of constant $p$.

\begin{q}\label{q3}
Determine an explicit constant $\b^1_p$ as a function of (constant) $p$ such that for $d=1$, 
\[
\lim_{n\to \infty} T(\gxp)=\b^1_pn.
\]
\end{q}

Our basic motivation has been to understand the constraint imposed on travel among random points by the restriction set of traversable edges which is chosen randomly independently of the geometry of the underlying point-set.  While the Erd\H{o}s-R\'enyi-Gilbert model is the prototypical example of a random graph, other models such as the Barab\'asi-Albert preferential attachment graph have received wide attention in recent years, due to properties (in particular, the distribution of degrees)  they share with real-world networks.  In particular, if the random graph one is traveling within is the flight-route map for an airline, the following questions may be the most relevant:
\begin{q}\label{pa}
If the preferential attachment graph is embedded randomly in the unit square (hypercube), what is the expected diameter?  What is the expected size of a minimum-length spanning tree?
\end{q}

Similarly, one could examine a combination of geometry and randomness in determining connections in the embedded graph.  Our methods already give something in this direction.  In particular, we can define $\gxpr$ as the intersection of the graphs $\gxp$ with the random geometric graph on the vertex set $\cX_n$, where a pair of points are joined by an edge if they are at distance $\leq r$.  Following our proof of Theorem \ref{tsp}, one sees that we find that

\begin{theorem}\label{geo}
If $d\geq2$, $p>0$ is constant, and $r=r(n)\geq n^{\e-1/d}$ for some $\e>0$, then 
\[
T(\gxpr)\sim \b^d_pn^{\frac{d-1} d} \qquad a.a.s.
\]
\end{theorem}
Of course, the ideas behind Question \ref{pa} and Theorem \ref{geo} could be considered together; note that Flaxman, Frieze and Vera \cite{FFV} considered a geometric version of a preferential attachment graph.

  The proof of Theorem \ref{t.alg} is relatively painless. We are reminded that Arora \cite{Ar} and Mitchell \cite{Mit} have described more sophisticated polynomial time algorithms that are asymptotically optimal even with the worst-case placing of the points. It would be interesting to see whether these algorithms can handle the random loss of edges.
\begin{q}
Do the methods of Arora and Mitchell allow efficiently approximation of the tour length through $\gxp$, when the embedding $\cX_n$ is \emph{arbitrary}?
\end{q}

\appendix
\section{Proof of \eqref{e.Ychernoff}}
Assume without loss of generality that we have scaled so that $\m=1$. Now $e^x\leq 1+x+x^2e^x$ when $x\geq 0$ and so for $\l>0,e^\l<1/\r$ we have
$$\E(e^{\l Y})\leq 1+\l+\l^2\brac{1+\frac{2}{(1-\r e^{\l})^3}}.$$
So, if $Z=Y_1+Y_2+\cdots+Y_n$ where $Y_1,Y_2,\ldots,Y_n$ are independent copies of $Y$,
\begin{align*}
\Pr(Z\geq n+\d n)&\leq e^{-\l(1+\d)n}\E(e^{\l Y})^n \\ &\leq e^{-\l(1+\d)n}\exp\set{\brac{\l+\l^2\brac{1+\frac{2}{(1-\r e^{\l})^3}}}n}\\ &\leq e^{-\l\d n}\exp\set{\l^2(1+2\e^{-3})}
\end{align*}
assuming that 
\beq{star}
e^\l\leq (1-\e)/\r.
\eeq
Now choose $\l=\d/(1+2\e^{-3})$ and $\e=\e(\d)$ such that \eqref{star} holds. Then
$$\Pr\brac{Z\geq n+\d n)}\leq \exp\set{-\frac{\d^2n}{2(1+2\e^{-3})}}.$$ 

\begin{thebibliography}{99}
\bibitem{Ar} S. Arora, Polynomial time approximation schemes for Euclidean Traveling Salesman and other geometric problems, {\em Journal of the Association for Computing Machinery} 45 (1998) 753-782.
\bibitem{BHH} J. Beardwood, J. H. Halton and J. M. Hammersley, The shortest path through many points, {\em Mathematical Proceedings of the Cambridge Philosophical Society} 55 (1959) 299-327.
\bibitem{BFF} B. Bollob\'as, T. Fenner and A.M. Frieze, An algorithm for finding Hamilton paths and cycles in random graphs, {\em Combinatorica} 7 (1987) 327-341.
\bibitem{FR} D. Fernholz and V. Ramachandran, The diameter of sparse random graphs, {\em Random Structures and Algorithms} 31 (2007) 482-516.
\bibitem{FFV} A. Flaxman, A.M. Frieze and J. Vera, A Geometric Preferential Attachment Model of Networks
{\em Internet Mathematics} 3 (2007) 187-205.
\bibitem{Fcycles} A.M. Frieze, On large matchings and cycles in sparse random graphs,
{\em Discrete Mathematics} 59 (1986) 243-256.
\bibitem{HK} M. Held and R.M. Karp,
A dynamic programming approach to sequencing problems, {\em SIAM Journal on Applied Mathematics} 10 (1962) 196-210.
\bibitem{JKV} A. Johansson, J. Kahn and V. Vu, Factors in Random Graphs,
{\em Random Structures and Algorithms} 33 (2008) 1-28.
\bibitem{K} R.M. Karp, Probabilistic Analysis of Partitioning Algorithms for the Traveling-Salesman Problem in the Plane, {\em Mathematics of Operations Research} 2 (1977) 209-244.
\bibitem{M11} A. Mehrabian, A Randomly Embedded Random Graph is Not a Spanner, In Proceedings of the 23rd Canadian Conference on Computational Geometry (CCCG 2011) (2011) 373-374.
\bibitem{MR} A. Mehrabian and N. Wormald, On the Stretch Factor of Randomly Embedded Random Graphs, to appear. 
\bibitem{Mit} J. Mitchell, Guillotine Subdivisions Approximate Polygonal Subdivisions:
A simple polynomial-time approximation scheme for geometric TSP, k-MST, and related problems, {\em SIAM Journal on Computing} 28 (1999) 1298-1309. 
\bibitem{P} M. Penrose, Random Geometric Graphs, Oxford University Press, 2003.
\bibitem{Po} L. P\'osa, Hamiltonian circuits in random graphs, {\em Discrete Mathematics} 14 (1976) 359-364.
\bibitem{RT} W. Rhee and M. Talagrand, A Sharp Deviation Inequality for the Stochastic Traveling Salesman Problem, in {\em The Annals of Probability} {17} (1989) 1--8.
\bibitem{RW} O. Riordan and N. Wormald, The diameter of sparse random graphs, {\em Combinatorics, Probability and Computing} 19 (2010) 835-926.
\bibitem{S} J. Michael Steele, Subadditive Euclidean functionals and nonlinear growth in geometric probability, {\em The Annals of Probability} 9 (1981) 365-376.
\end{thebibliography}
\end{document}